\documentclass[a4paper,11pt]{article}

\usepackage{amsmath,latexsym}
\usepackage{amsthm}
\usepackage{amssymb}
\usepackage{euler}
\usepackage{graphics, graphicx}
\usepackage{amsfonts}
\usepackage[nesting]{hyperref}
\usepackage{longtable}
\usepackage{color}
\usepackage{colortbl}
\usepackage{pdflscape}
\providecommand{\keywords}[1]{\textbf{\textit{Keywords:}} #1}

\newtheorem{theorem}{Special Theorem}

\newtheorem{corollary}[theorem]{Corollary}
\newtheorem{lemma}[theorem]{Lemma}
\newtheorem{example}[theorem]{Example}
\newtheorem{remark}[theorem]{Remark}

{\bf}{\it}

\def\K{\mathbf{K}}

\def\R{\mathbb{R}}
\def\N{\mathbb{N}}

\def\B{\mathcal{B}}

\def\Mo{\mathrm{M}}
\def\L{\mathrm{L}}

\def\ll{\ell}

\newcommand{\fin}{\hspace \fill
 $\square$  \par \bigskip}

\begin{document}

\title{Revisiting several problems and algorithms in continuous location with $\ell_p$ norms}

\author{V\'ictor Blanco, Justo Puerto, Safae El Haj Ben Ali}

\date{}

\maketitle

\begin{abstract}
This paper addresses the general continuous single facility location problems in finite dimension spaces under possibly different $\ell_p$ norms in the demand points. We analyze the difficulty of this family of problems and revisit convergence properties of some well-known algorithms. The ultimate goal is to provide a common approach to solve the family  of continuous $\ell_p$  ordered median location problems in dimension $d$ (including of course the $\ell_p$ minisum or Fermat-Weber location problem for any $p\ge 1$). We prove that this approach has a polynomial worse case complexity for monotone lambda weights and can be also applied to constrained and even non-convex problems.\\
\end{abstract}

\section{Introduction}

Location analysis is a very active topic within the Operations Research community. It  has giving rise to a number of nowadays standard optimization problem some of them in the core of modern mathematical programming. One of its branches is continuous location a family of models directly related to important areas of mathematics such as linear and non-linear programming, convex analysis and global optimization (see e.g. \cite{DH02} and the references therein). It is widely agreed that modern continuous location started with the paper by Weber (1909) \cite{weber} who first considers the minimization of weighted sums of distances as an economical goal to locate industries. This problem is currently known as \textit{Fermat-Weber}, also because of the three points Fermat problem (s. XVII) firstly solved by Torricelli in 1659. The algorithmic part of this history starts at 1937 with the paper by Weiszfeld \cite{weiszfeld} who proposed an iterative gradient type algorithm  to find or to approximate the solutions of the above mentioned Fermat-Weber problem.

For several decades this algorithm remains forgotten but in 1973 Kuhn \cite{K1973} rediscovered it and proved its convergence, under some conditions, in the Euclidean case. One year later Katz \cite{Katz74} gives another convergence result. Several years later, a number of authors considered the weighted minisum problem under different norms mainly $\ell_p$ or polyhedral (see e.g. \cite{DH02} for a detailed literature review) and Chandrasekaran and Tamir \cite{CT1989} raise several interesting questions concerning resolubility of Weiszfeld algorithm.
Eventually, starting in the nineties, several authors were very interested in proving the convergence of some modifications of the Weiszfeld algorithm, usually called \textit{modified Weiszfeld} or \textit{generalized iterative procedure} for minisum location problems.

The convergence for Euclidean distances ($p = 2$) was  studied later by \cite{K1973,Katz74}, among others. Since then, we can find in the literature many references concerning this algorithm, as for instance the generalization to $\ell_p$ distances with $p \in [1, 2]$ \cite{MoVe79} or the analysis of its local and global convergence \cite{Bri92,Bri93,Bri95,BCC98}. Also, these results were extended to more general problems: on Banach spaces \cite{Eckhardt80,PuRCh99,PuRCh06}, on the sphere \cite{Zhang03}, with regional demand \cite{Chen01,VRE08}, with sets as demand facilities and using closest Euclidean distances \cite{Bri02} or with radial distances \cite{Chen84a,Chen84b,Cooper68,Katz69,Morris81}. In addition, one can find in the literature papers where the convergence is accelerated using alternative step sizes \cite{Drezner92,Drezner95,Harris76,Katz74,Ostresh}
or some related properties concerning the termination of the algorithm in any of the demand points after a finite number of iterations \cite{Bri95,BC98,Bri03,CCM02,CT1989,K1973}.

The influence of Weiszfeld algorithm in Location Analysis has been rather important very likely due to its very easy implementation. For several years, it was a very effective method to solve minisum continuous location problems, even though its theoretical convergence was not proven.  Thus, locators have devoted a lot of effort to prove its convergence. The global convergence result of this algorithm for $\ell_p$ $p\in [1,2]$ was proved in \cite{Bri93}  and  recently  \cite{RChV12} has given a proof to close the cases ($p>2$) that were not yet justified. This has been an important effort from a mathematical point of view. Nevertheless, pursuing this goal locators did not focus on the origin of the problem, namely to search for alternative, efficient algorithms to solve the $\ell_p$ minisum and some more general families of location problems.

The situation is even harder if we consider a more general family of location problems that have attracted a lot of attention in the field in the last years, namely continuous ordered median location problems \cite{NP05}. Ordered median problems represent as special cases
nearly all classical objective functions in location theory, including
the Median, CentDian, center and k-centra. More precisely,
the 1-facility ordered median problem  can be formulated
as follows: A vector of weights $(\lambda_1,\ldots,\lambda_n)$ is given. The
problem is to find a location for a facility that minimizes the
weighted sum of distances where the distance to the closest point
to the facility is multiplied by the weight $\lambda_n$, the distance to the second closest, by $\lambda_{n-1}$, and so on. The distance to the farthest point is multiplied by $\lambda_1$.
Many location problems can be formulated as the ordered 1-median
problem by selecting appropriate weights. For example,
the vector for which all $\lambda_i= 1$ is the unweighted 1-median problem,
the problem where $\lambda_1= 1$ and all others are equal to zero is
the 1-center problem, the problem where $\lambda_1=\ldots=\lambda_k=1$ and all others are equal to zero is the $k$-centrum. Minimizing the range of distances is
achieved by $\lambda_1=1$, $\lambda_n= -1$ and all others are zero. Despite its full generality, the main drawback of this framework is the difficulty of solving the problems with a unified tool. There have been some successful approaches that are now available whenever the framework space is either discrete (see \cite{BDNP05,MNPV08}) or a network (see \cite{KNP}, \cite{KNPT},   \cite{NiPu99} or \cite{PuTa05}). Nevertheless, the continuous case has been, so far, only partially covered even under the additional hypothesis of convexity. There have been some attempts to overcome this drawback and there are nowadays some available methodologies to tackle these problems, at least in the plane and with Euclidean norm. In Drezner \cite{drezner2007} and Drezner and Nickel \cite{DN09-01,DN09-02} the authors present two different approaches. The first one uses a geometric branch and bound method based on triangulations (BTST) and the second one on a D-C decomposition for the objective function that allow solving problems on the plane.  Espejo et al \cite{ERChV09}  also address the unconstrained   convex ordered median location problem on the plane and Rodriguez-Chia et al. \cite{ERChD10} attacks the  $k$-centrum problem using geometric arguments and developing a better algorithm applicable only for that unconstrained problem on the plane and Euclidean distances. More recently, Blanco et al. \cite{BPS12} have presented a new methodology based on a  hierarchy of SDP relaxations that can be used to solve (approximate) the optimal solutions of the general ordered  median location problems which main drawback is the size of the SDP objects that have to be used to get good accuracy in high dimension.

The above discussion points out that there exists a lack of a unified resolution approach to those problems as well as effective algorithms for the general cases. Our goal in this paper  is to design a common approach to solve  the family  of continuous $\ell_p$  ordered median location problems in dimension $d$ (including of course the $\ell_p$ minisum or Fermat-Weber location problem for any $p\ge 1$). We prove that this approach has a polynomial worse case complexity for monotone lambda weights and can be also applied to constrained and to approximate even non-convex problems. Thus, providing a unifying new algorithmic paradigm for this class of location problems. First, for  convex location problems it avoids the drawback of limit convergence proven for the Weiszfeld type algorithms. Then, it can be applied to any convex ordered median problem, even with mixed norms, in any dimension and with rather general convex constraints.  Moreover, we show an explicit reformulation of these problems as  SDP  problems which enables the usage of standard free source solvers (SEDUMI, SDPT3,...) to solve them up to any degree of accuracy. Finally, we also show how to adapt this approach to approximate up to any degree of accuracy non-convex constrained location problems using a hierarchy of convergent relaxed problems.

The paper is organized in 5 sections. In Section \ref{s:2} we provide a compact representation, valid for any unconstrained convex ordered location problem, by means of a new formulation that reduces these problems to semidefinite problems. This approach allows us to ensure that all these problems are polynomially solvable in finite dimension.  Section \ref{s:constrained} is devoted to extend the results of Section \ref{s:2} to the case of constrained problems under the condition of SDP-representability. Then, we handle the general case of non-convex constrained ordered median location problem for which we construct a hierarchy of SDP relaxations that converges to the optimal solution of the original problem. Our Section \ref{s:4} is devoted to the computational experiments. We report results in four different problems type, namely minisum (Weber), minimax (center), $k$-centrum (minimizing the sum of the $k$-largest distances) and general ordered median problems. The paper ends, in Section \ref{s:5}, with some conclusions and an outlook for further research .

\section{A compact representation of the convex ordered median problem\label{s:2}}

In this section we present the convex ordered median problem in dimension $d$
where the distances are measured with a general $\ell_{\tau}$-norm being $\tau \in \mathbb{Q}$. We are given a set of demand points $S = \{a_{1},  . . .  , a_{n}\}$ and two sets of
scalars $\Omega := \{\omega_{1},  . . .  , \omega_{n}\}$, $\omega_i\ge 0,\; \forall\; i\in \{1, \ldots, n\} $ and $\Lambda := \{\lambda_{1}
, . . . , \lambda_{n}\}$ where $\lambda_{1}\geq . . . \geq\lambda_{n}\ge 0$.
The elements $\omega_{i}$ are weights corresponding to the importance given
to the existing facilities $a_{i}, i \in \{1,  . . .  ,n\}$ and depending
on the choice of the elements of $\Lambda$ we get different classes of problems.  We denote by $\mathcal{P}_n$ the set of permutations of the first $n$ natural numbers.

Given a permutation $\sigma\in \mathcal{P}_n$ satisfying
$$
\omega_{\sigma(1)}\|x-a_{\sigma(1)}\|_{\tau}\geq\ldots\geq\omega_{\sigma(n)}\|x-a_{\sigma(n)}\|_{\tau},
$$
\noindent the unconstrained  ordered median problem (see \cite{NP05}) consists of

\begin{equation}\label{pb1}
\min_{x\in\R^{d}}\sum_{i=1}^{n}\lambda_{i}\omega_{\sigma(i)}\|x-a_{\sigma(i)}\|_{\tau}.
\end{equation}

We start by showing a compact reformulation of the above problem  that will be later useful in our approach.

\begin{theorem} \label{th:rep}
Let $\tau=\frac{r}{s}$ be such that $r,s\in\N\setminus\{0\}$, $r>s$ and $\gcd(r,s)=1$. For any set of lambda weights satisfying $\lambda_{1}\geq . . . \geq\lambda_{n}$, Problem \eqref{pb1} is equivalent to
      \begin{eqnarray}
        \min& \displaystyle \sum_{k=1}^{n}v_{k}+\sum_{i=1}^{n}w_{i}  & \label{genpb} \\
        s.t &  v_{i}+w_{k}\geq\lambda_{k}z_{i} , & \forall i,k=1,...,n, \\
         & y_{ij}-x_j+a_{ij}\ge 0,& i=1,\ldots,n,\; j=1,...,d, \label{eq:n1-1}\\
         & y_{ij}+x_j-a_{ij}\ge 0,& i=1,\ldots,n,\; j=1,...,d,\label{eq:n1-2}\\
        & y_{ij}^r\leq u_{ij}^{s}z_{i}^{r-s},& i=1,\ldots,n,\; j=1,...,d,, \label{eq:n1-3}\\
        &\omega_{i}^{\frac{r}{s}}\sum_{j=1}^d u_{ij}\le z_i,& i=1,\ldots,n,\; \\
        & u_{ij}\ge 0,&  i=1,\ldots,n,\; j=1,\ldots,d. \label{eq:n1-4}
      \end{eqnarray}
\end{theorem}
\begin{proof}

Because of the condition $\lambda_{1}\geq . . . \geq\lambda_{n}$, Problem (\ref{pb1}) can be equivalently written as

\begin{equation}\label{pb2:ini}
        \displaystyle\min_{x\in\R^{d}}\max_{\sigma\in \mathcal{P}_n}\quad \displaystyle\sum_{i=1}^{n}\lambda_{i}\omega_{\sigma(i)}\|x-a_{\sigma(i)}\|_{\tau},
\end{equation}

Let us introduce auxiliary variables $z_{i}$, $i=1,\ldots,n$ to which we impose that $ z_{i}\geq\omega_{i}\|x-a_{i}\|_{\tau}$, to model the problem in a convenient form.
Now, for any permutation $\sigma \in \mathcal{P}_n$, let $z_{\sigma}=(z_{\sigma(1)},\ldots,z_{\sigma(n)})$. Moreover, let us denote by $(\cdot)$ the permutation that sorts any vector in nonincreasing sequence, i.e. $z_{(1)}\ge z_{(2)}\ge \ldots \ge z_{(n)}$.
Using that $\lambda_{1}\geq...\geq\lambda_{n}$ and since  $z_i\ge 0$, for all $i=1,\ldots,n$ then
$$
\displaystyle \sum_{i=1}^{n}\lambda_{i}z_{(i)}=\max_{\sigma\in \mathcal{P}_n}\displaystyle \sum_{i=1}^{n}\lambda_{i}z_{\sigma(i)}.
$$

The permutations in $\mathcal{P}_n$ can be represented by the following binary variables
$$
p_{ik}=\left\{
         \begin{array}{ll}
           1, & \mbox{ if } z_{i} \mbox{ goes in position } k, \\
           0, & \mbox{ otherwise},
         \end{array}
       \right.
$$

imposing that they verify the following constraints:

\begin{equation}\label{sys}
\left\{
  \begin{array}{ll}
    \displaystyle \sum_{i=1}^{n}p_{ik}=1, & \forall k=1,...,n, \\
    \displaystyle \sum_{k=1}^{n}p_{ik}=1, & \forall i=1,...,n.
  \end{array}
\right.
\end{equation}

Next, combining the two sets of variables we obtain that the objective function of \eqref{pb2:ini} can be equivalently written as::
\begin{equation}\label{sys0}
    \left\{
      \begin{array}{ll}
        \displaystyle \sum_{i=1}^{n}\lambda_{i}z_{(i)} =&\max\displaystyle \sum_{i=1}^{n}\sum_{k=1}^{n}\lambda_{k}z_{i}p_{ik} \\
        & s.t \quad \displaystyle \sum_{i=1}^{n}p_{ik}=1, \; \forall k=1,...,n, \\
        & \qquad \displaystyle \sum_{k=1}^{n}p_{ik}=1, \; \forall i=1,...,n, \\
        & \qquad p_{ik}\in\{0,1\}.
      \end{array}
    \right.
\end{equation}
Now, we point out that for fixed $z_{1},...,z_{n}$, the above problem is an assignment problem and its constraint matrix is totally unimodular, so that solving a continuous relaxation of the problem always yields an integral solution vector \cite{AMO1993}, and thus a valid permutation.
Moreover, the dual of the linear programming relaxation of (\ref{sys0}) is strong and also gives the value of the original binary formulation of (\ref{sys0}).

Hence, for any vector $z\in \mathbb{R}^n$, by using the dual of the assignment problem (\ref{sys0}) we obtain the following expression
\begin{equation}\label{sys2}
    \left\{
       \begin{array}{ll}
         \displaystyle \sum_{i=1}^{n}\lambda_{i}z_{(i)}=&  \min\displaystyle \sum_{k=1}^{n}v_{k}+\sum_{i=1}^{n}w_{i} \\
         &  s.t \quad v_{i}+w_{k}\geq\lambda_{k}z_{i},\; \forall i,k=1,...,n.
       \end{array}
     \right.
\end{equation}

Finally, we replace (\ref{sys2}) in (\ref{pb2:ini}) and we get

\begin{equation}\label{genpb1}
    \left\{
      \begin{array}{ll}
        \min\displaystyle \sum_{k=1}^{n}v_{k}+\sum_{i=1}^{n}w_{i} \\
        s.t \quad v_{i}+w_{k}\geq\lambda_{k}z_{i} , & \forall i,k=1,...,n, \\
        \qquad z_{i}\geq\omega_{i}\|x-a_{i}\|_{\tau},& i=1,...,n.
      \end{array}
    \right.
\end{equation}

It remains to prove that each inequality $z_{i}\geq\omega_{i}\|x-a_{i}\|_{\tau},\;i=1,...,n$ can be replaced by the system:
\begin{eqnarray*}
 y_{ij}-x_j+a_{ij}\ge 0,&  j=1,...,d. \\
 y_{ij}+x_j-a_{ij}\ge 0,&  j=1,...,d.\\
y_{ij}^r\leq u_{ij}^{s}z_{i}^{r-s},&  j=1,...,d. \\
\omega_{i}^{\frac{r}{s}}\sum_{j=1}^d u_{ij}\le z_i,& \\
 u_{ij}\ge 0,& \forall\;  j=1,\ldots,d.
\end{eqnarray*}

Indeed, set $\rho=\frac{r}{r-s}$, then $\frac{1}{\rho}+\frac{s}{r}=1$. Let $(\bar x,\bar z_i)$ fulfill the inequality  $z_{i}\geq\omega_{i}\|x-a_{i}\|_{\tau}$. Then we have
\begin{eqnarray}
  \omega_{i}\|\bar x-a_{i}\|_{\tau}\leq \bar z_{i} &\Longleftrightarrow& \omega_{i}\left(\sum_{j=1}^{d}|\bar x_{j}-a_{ij}|^{\frac{r}{s}}\right)^{\frac{s}{r}}\leq \bar z_{i}^{\frac{s}{r}}\bar z_{i}^{\frac{1}{\rho}} \nonumber \\
   &\Longleftrightarrow& \omega_{i}\left(\sum_{j=1}^{d}|\bar x_{j}-a_{ij}|^{\frac{r}{s}}\bar z_{i}^{\frac{r}{s}(-\frac{r-s}{r})}\right)^{\frac{s}{r}}\leq \bar z_{i}^{\frac{s}{r}}\nonumber \\
&\Longleftrightarrow& \omega_{i}^{\frac{r}{s}}\sum_{j=1}^{d}|\bar x_{j}-a_{ij}|^{\frac{r}{s}}\bar z_{i}^{-\frac{r-s}{s}}\leq \bar z_{i} \label{eq1}
\end{eqnarray}
Then (\ref{eq1}) holds if and only if $\exists u_{i}\in\R^d$, $u_{ij}\ge 0,\; \forall j=1,...,d$ such that
$$
   |\bar x_{j}-a_{ij}|^{\frac{r}{s}}\bar z_{i}^{-\frac{r-s}{s}}\leq u_{ij},\quad\mbox{ satisfying }\quad \omega_{i}^{\frac{r}{s}}\sum_{j=1}^{d}u_{ij}\leq \bar z_{i}, $$
or equivalently,
\begin{equation} \label{eq:n1} |\bar x_{j}-a_{ij}|^{r}\leq u_{ij}^{s}\bar z_{i}^{r-s},\quad \omega_{i}^{\frac{r}{s}}\sum_{j=1}^{d}u_{ij}\leq \bar z_{i}.
\end{equation}
Set $\bar y_{ij}=|\bar x_j-a_{ij}|$ and $\bar u_{ij}=|\bar x_j-a_{ij}|^{\tau} \bar z_i^{-1/\rho}$. Then, clearly $(\bar x,\bar z_i,\bar y,\bar u)$ satisfies (\ref{eq:n1-1})-(\ref{eq:n1-4}).

Conversely, let $(\bar x,\bar z_i,\bar y,\bar u)$ be a feasible solution of (\ref{eq:n1-1})-(\ref{eq:n1-4}). Then, $\bar y_{ij}\ge |\bar x_{ij}-a_{ij}|$ for all $i,j$ and by (\ref{eq:n1-3}) $\bar u_{ij}\ge \bar y_{ij}^{(r/s)}z_i^{-\frac{r-s}{s}}\ge |\bar x_j-a_{ij}|^{\tau} \bar z_i^{-\frac{r-s}{s}}$. Thus,
$$ \omega_i^{r/s}\sum_{j=1}^d |\bar x_j-a_{ij}|^{r/s} \bar z_i^{-\frac{r-s}{s}} \le \omega_i^{r/s} \sum_{j=1}^d \bar u_{ij} \le \bar z_i,$$
which in turns implies that $\omega_i^{r/s}\sum_{j=1}^d |\bar x_j-a_{ij}|^{r/s} \le \bar z_i \bar z_i^{\frac{r-s}{s}}$ and hence, $\omega_i\|\bar x-a_i\| \le \bar z_i$.

\fin
\end{proof}

Problem (\ref{genpb}) -(\ref{eq:n1-4}) is an exact representation of Problem (\ref{pb1}) in any dimension and for any $\ell_{\tau}$-norm.

For the case $\tau=1$, the above problem reduces to a linear programming problem.
\begin{corollary}
If $\tau=1$ the reformulation given by Problem (\ref{pb1}) is
\begin{equation}\label{genpb-l1}
    \left\{
      \begin{array}{ll}
        \min\displaystyle \sum_{k=1}^{n}v_{k}+\sum_{i=1}^{n}w_{i} \\
        s.t \quad v_{i}+w_{k}\geq\lambda_{k}z_{i} , & \forall i,k=1,...,n, \\
        \qquad z_{i}\geq\omega_{i}\sum_{j=1}^d u_{ij},& i=1,...,n,\\
        \qquad x_j-a_{ij}\le u_{ij} & i=1,...,n,\; j=1,\ldots,d,\\
        \qquad -x_j+a_{ij}\le u_{ij} & i=1,...,n,\; j=1,\ldots,d.
      \end{array}
    \right.
\end{equation}
\end{corollary}

The reader may observe that the representation given in Theorem \ref{th:rep} is new and different from the one used in \cite{BPS12}. On the one hand, this new formulation is more efficient than the  one presented in \cite{BPS12} and specially tailored for the case of non-increasing monotone lambda weights, see e.g.  \cite[Lemma 8]{BPS12}. For the sake of readability we include it in the following.

Let \begin{equation}\label{kC0}
S_k(x):=\sum_{j=1}^k z_{(j)},
\end{equation}
where $z_{(j)}$ is such that $z_{(1)}\ge \ldots \ge z_{(n)}$.
That formulation applied to the setting of this paper reads as:

\begin{eqnarray}
 \min & \sum_{k=1}^n (\lambda_k-\lambda_{k+1}) S_k(x) \label{pro:lam-mon2} \\
 &  t_k+r_{kj}\ge z_j(x), \quad j,k=1,\ldots,n, \nonumber \\
 & r_{kj}\ge 0,\quad j,k=1,\ldots,n, \nonumber \\
 &  z_{j}\geq\omega_{j}\|x-a_{j}\|_{\tau}, \quad j=1,...,n. \nonumber
\end{eqnarray}

It is easy to see that formulation \eqref{pro:lam-mon2} has  $O(n^2+d)$ variables and $O(n^2)$ constraints whereas the new one written in similar terms as presented  in \eqref{genpb1} has $O(n+d)$ variables and $O(n^2)$ constraints.

Our goal is to show that for any  $\tau \in \mathbb{Q}$, Problem (\ref{genpb}) -(\ref{eq:n1-4}) also admits a compact formulation within another easy class of polynomially solvable mathematical programming problems. In order to get that we need to prove a technical lemma. Let $\#A$ denote the cardinality of the set $A$.

\begin{lemma} \label{le:n2}
Let $\tau=\frac{r}{s}$ be such that $r,s\in\N\setminus\{0\}$, $r>s$ and $\gcd(r,s)=1$.\\
Let $x,\;u\mbox{ and }t$ be non negative and satisfying
\begin{equation} \label{eq:ineq1}
x^{r}\leq u^{s}t^{r-s}.
\end{equation}
Assume that $2^{k-1} < r\leq2^{k}$ where $k\in\N\setminus\{0\}$ such that
\begin{equation}\label{00}
x^{2^{k}}\leq u^{s}t^{r-s}x^{2^{k}-r},
\end{equation}
and
\begin{eqnarray}
  s &=& \alpha_{k-1}2^{k-1}+\alpha_{k-2}2^{k-2}+\ldots+\alpha_{1}2^{1}+\alpha_{0}2^{0} \label{1},\\
  r-s &=& \beta_{k-1}2^{k-1}+\beta_{k-2}2^{k-2}+\ldots+\beta_{1}2^{1}+\beta_{0}2^{0} \label{2},\\
  2^{k}-r &=& \gamma_{k-1}2^{k-1}+\gamma_{k-2}2^{k-2}+\ldots+\gamma_{1}2^{1}+\gamma_{0}2^{0} \label{3},
\end{eqnarray}
where $\alpha_{i},\;\beta_{i},\;\gamma_{i}\in\{0,1\}$.

Then, if $(x,t,u)$ is a feasible solution of (\ref{eq:ineq1}) there exists $w$ such that  either

\begin{description}
\item[1.]  $(x,t,u,w)$ is a solution of System (\ref{sys1}), if $\alpha_i+\beta_i+\gamma_i=1$, for all $0<i<k-1$.
\begin{equation}\label{sys1}
    \left\{
      \begin{array}{ll}
        w_{1}^{2} &\leq u^{\alpha_0} t^{\beta_0} x^{\gamma_0}, \\
        w_{i+1}^{2}& \le w_i u^{\alpha_i} t^{\beta_i} x^{\gamma_i},\; i=1,\ldots,k-2   \\
        x^{2} &\leq w_{k-1}d,
      \end{array}
    \right.
\end{equation}
where $d=\left\{ \begin{array}{ll} w_{k-2} & \mbox{if } \alpha_{k-1}\beta_{k-1}+\gamma_{k-1}=0 \\  u^{\alpha_{k-1}} t^{\beta_{k-1}} u^{\gamma_{k-1}} & \mbox{if } \alpha_{k-1}\beta_{k-1}+\gamma_{k-1}=1 \end{array} \right.$
\item[2.] $(x,t,u,w)$ is a solution of System (\ref{sys1.2}), if  there exist $i_j$ and $i_{l(j)}$, $j=1,\ldots,c$ such that: \\
\textit{1. }  $0<i_1<i_2<\ldots<i_c\le k-2$,\\
\textit{2. }  $i_j<i_{l(j)}<i_{j+1}$,\\
\textit{3. } $\alpha_{i_j}+\beta_{i_j}+\gamma_{i_j}=3$,
  $\alpha_{i_{l(j)}}+\beta_{i_{l(j)}}+\gamma_{i_{l(j)}}=0$ and  $\alpha_{h}+\beta_{h}+\gamma_{h}=2$ for $h=i_j+1,\ldots,i_{l(j)-1}.$
\par \medskip

 \begin{equation}\label{sys1.2}
   \hspace*{-2cm} \left\{
      \begin{array}{ll}
        w_{1}^{2} &\leq u^{\alpha_0} t^{\beta_0} x^{\gamma_0}, \\
        w_{i+1}^{2}& \le w_i u^{\alpha_i} t^{\beta_i} x^{\gamma_i},\; i\in \{1,\ldots,i_{1}-1\} \\
        &\hspace*{-2.8cm}--------\mbox{for each } j=1,\ldots,c---------\\
        w_{\theta(j)}^{2}&\leq ut,\\
        w_{\theta(j)+1}^{2}&\leq w_{\theta(j)-1}x\\
        &\hspace*{-2.8cm}\left. \begin{array}{ll}
        w_{\theta(j)+2*s}^{2}&\leq w_{\theta(j)+2(s-1)}a_{i_j+s}  \\
        w_{\theta(j)+2*s+1}^{2}&\leq w_{\theta(j)+2s-1}b_{i_j+s}
               \end{array}
        \right\},\; \begin{array}{l} s=1,\ldots,i_{l(j)}-i_j-1\;and \\ a_{i_j+s}+b_{i_j+s}=u^{\alpha_{i_j+s}}t^{\beta_{i_j+s}}x^{\gamma_{i_j+s}},  \end{array} \\
        w_{\theta(j)+2(i_{l(j)}-i_j)}^{2}&\leq w_{\theta(j)+2(i_{l(j)}-i_j-1)}w_{\theta(j)+2(i_{l(j)}-i_j-1)+1},\; {\scriptsize \left\{\begin{array}{l} if \;m-1>\\ \theta(j)+2(i_{l(j)}-i_j-1)\end{array},\right.} \\
        w_{\theta(j)+2(i_{l(j)}-i_j)+s}^{2}& \leq w_{\theta(j)+2(i_{l(j)}-i_j)+s-1} u^{\alpha_{i_{l(j)}+s}} t^{\beta_{i_{l(j)}+s}} x^{\gamma_{i_{l(j)}+s}},\; {\scriptsize \left\{\begin{array}{l} \mbox{for all } s=1,\ldots,\\ \quad i_{j+1}-i_{l(j)}-1\end{array}\right.} \\
        &\hspace*{-2.8cm}------------------------------\\
        x^{2} &\leq w_{m}d.
      \end{array}
    \right.
\end{equation}
where $\theta=(\theta(j))_{j=1}^c$ such that $\theta(j)=2\#\{i:\alpha_{i}+\beta_{i}+\gamma_{i}\geq2,\; 1<i\leq i_{j}\}+\#\{i:\alpha_{i}+\beta_{i}+\gamma_{i}\leq1,\; 1<i\leq i_{j}\}$ for $j=1,...,c$,
$m=1+2\#\{i:\alpha_{i}+\beta_{i}+\gamma_{i}\geq2, 1<i<k\}+\#\{i:\alpha_{i}+\beta_{i}+\gamma_{i}=1, 1<i<k\}\le 2k$ and $d=\left\{ \begin{array}{ll} w_{m-1} & \mbox{if } \alpha_{k-1}+\beta_{k-1}+\gamma_{k-1}=0 \\  u^{\alpha_{k-1}} t^{\beta_{k-1}} u^{\gamma_{k-1}} & \mbox{if } \alpha_{k-1}+\beta_{k-1}+\gamma_{k-1}=1 \end{array} \right.$.
\end{description}
Conversely, if $(x,t,u,w)$ is a solution of (\ref{sys1}) or (\ref{sys1.2}) then $(x,t,u)$ is a feasible solution of (\ref{eq:ineq1}).
\end{lemma}

\begin{proof}.
To get the expressions of any of the systems (\ref{sys1}) or (\ref{sys1.2}), we discuss the decomposition (\ref{1}), (\ref{2}), (\ref{3}) of $s,\;r-s\mbox{ and }2^{k}-r$ in the basis $B=\{2^{l}\},\;l=0,...,k-1$.

Since $2^{k}=2^{k}-r+(r-s)+s$, we observe that (\ref{1})+(\ref{2})+(\ref{3}) gives a decomposition of $2^{k}$ in power of $2$ summands with  coefficients less than or equal than 3. Namely,
\begin{equation}\label{4}
2^{k}=(\alpha_{k-1}+\beta_{k-1}+\gamma_{k-1})2^{k-1}+\ldots+(\alpha_{1}+\beta_{1}+\gamma_{1})2^{1}+(\alpha_{0}+\beta_{0}+\gamma_{0})2^{0}.
\end{equation}
We discuss  two cases depending on the parity of $r$.

\noindent If $r$ is even then $s$ is odd since $\gcd(r,s)=1$, thus $r-s$ is odd and $2^{k}-r$ is even.

\noindent If $r$ is odd then $s$ can be odd or even; if $s$ is odd then $r-s$ is even and $2^{k}-r$ is odd; otherwise $r-s$ is odd and $2^{k}-r$ is odd.

From the above discussion, we observe that there are always two odd and one even numbers in the triplet $(s,r-s,2^{k}-r)$. Therefore, we conclude  that $$\alpha_{0}+\beta_{0}+\gamma_{0}=2.$$

On the other hand, since $\displaystyle\sum_{i=0}^{k-1}2^{i}=2^{k}-1$, then another representation of $2^k$ is:
\begin{equation}\label{5}
2^{k}=1.2^{k-1}+1.2^{k-2}+\ldots+1.2^{1}+2.2^{0}.
\end{equation}

Considering the fact that  (\ref{5}) and (\ref{4}) are two representations of $2^k$, by equating coefficients, we deduce some properties of the sums $(\alpha_{i}+\beta_{i}+\gamma_{i}),\;i=1,...,k-1.$

\begin{description}

\item[$\bullet$] First of all, we observe that $\alpha_{k-1}+\beta_{k-1}+\gamma_{k-1}$ can only assume the values $0$ or $1$.
  \item[$\bullet$] Second, since $\alpha_{0}+\beta_{0}+\gamma_{0}=2$ then it implies that $\alpha_{1}+\beta_{1}+\gamma_{1}=1\; or\; 3$, otherwise if $\alpha_{1}+\beta_{1}+\gamma_{1}=0\; or\; 2, $ then we will get $0.2^{1}=0$ or $2.2^{1}=2^{2}$ which means that we can not recover the term $2^{1}$ and then we will not get the decomposition as in (\ref{5}).
\item[$\bullet$] Third,  let $i_{0}$ be the first index, counting in a decreasing order from $k-1$ to 1, so that $\alpha_{i_{0}}+\beta_{i_{0}}+\gamma_{i_{0}}\neq1$ and $\forall\;i,$ $i_{0}<i\leq k-1$ we have $\alpha_{i}+\beta_{i}+\gamma_{i}=1$. Then three cases can occur:\\
\begin{enumerate}
\item if $\alpha_{i_{0}}+\beta_{i_{0}}+\gamma_{i_{0}}=3$, then
\begin{eqnarray*}
 2^{k} &=& 1.2^{k-1}+...+1.2^{i_{0}+1}+3.2^{i_{0}}+(\alpha_{i_{0}-1}+\beta_{i_{0}-1}+\gamma_{i_{0}-1})2^{i_{0}-1}+...+2.2^{0}, \\
       &=& 1.2^{k-1}+...+2.2^{i_{0}+1}+1.2^{i_{0}}+(\alpha_{i_{0}-1}+\beta_{i_{0}-1}+\gamma_{i_{0}-1})2^{i_{0}-1}+...+2.2^{0}, \\
       &\vdots&\\
       &=& 2.2^{k-1}+...+0.2^{i_{0}+1}+1.2^{i_{0}}+(\alpha_{i_{0}-1}+\beta_{i_{0}-1}+\gamma_{i_{0}-1})2^{i_{0}-1}+...+2.2^{0},
\end{eqnarray*}
which it is not possible and therefore it implies that $\alpha_{i_{0}}+\beta_{i_{0}}+\gamma_{i_{0}}=2\; or\;0.$\\
\item if $\alpha_{i_{0}}+\beta_{i_{0}}+\gamma_{i_{0}}=2$, then
\begin{eqnarray*}
 2^{k} &=& 1.2^{k-1}+...+1.2^{i_{0}+1}+2.2^{i_{0}}+(\alpha_{i_{0}-1}+\beta_{i_{0}-1}+\gamma_{i_{0}-1})2^{i_{0}-1}+...+2.2^{0}, \\
       &=& 1.2^{k-1}+...+2.2^{i_{0}+1}+0.2^{i_{0}}+(\alpha_{i_{0}-1}+\beta_{i_{0}-1}+\gamma_{i_{0}-1})2^{i_{0}-1}+...+2.2^{0}, \\
       &\vdots&\\
       &=& 2.2^{k-1}+...+0.2^{i_{0}+1}+0.2^{i_{0}}+(\alpha_{i_{0}-1}+\beta_{i_{0}-1}+\gamma_{i_{0}-1})2^{i_{0}-1}+...+2.2^{0},
\end{eqnarray*}
which  again it is not possible and therefore it implies that $\alpha_{i_{0}}+\beta_{i_{0}}+\gamma_{i_{0}}=0.$\\
From the above two cases, we summarize that the first sum $\alpha_{i_0}+\beta_{i_0}+\gamma_{i_0}\neq1$ must be necessarily $\alpha_{i_0}+\beta_{i_0}+\gamma_{i_0}=0$. Based on this we consider the only possible case.
\item if $\alpha_{i_0}+\beta_{i_0}+\gamma_{i_0}=0$, then  it must exist $ i_{1}<i_0$, satisfying that $\alpha_{i_{1}}+\beta_{i_{1}}+\gamma_{i_{1}}=3$ and such that for all $k$,  $i_1<k\leq i_{0}$, $\alpha_{k}+\beta_{k}+\gamma_{k}=2$. Indeed,\\
\begin{enumerate}
\item if $\alpha_{i_0-1}+\beta_{i_0-1}+\gamma_{i_0-1}=1$, then
\begin{eqnarray*}
 2^{k} &=& 1.2^{k-1}+...+1.2^{i_0+1}+0.2^{i_0}+(\alpha_{i_0-1}+\beta_{i_0-1}+\gamma_{i_0-1})2^{i-1}+...+2.2^{0}, \\
       &=& 1.2^{k-1}+...+1.2^{i_0+1}+0.2^{i_{0}}+1.2^{i_0-1}+...+2.2^{0}.
\end{eqnarray*}
Hence, since the sums $\alpha_{j}+\beta_{j}+\gamma_{j}\le 3$ for all $j$, one cannot  recover the sum $2^k$ in (\ref{5}) and the representation of $2^k$ would be wrong.\\
\item if $\alpha_{i_0-1}+\beta_{i_0-1}+\gamma_{i_0-1}=3$, then
 \begin{eqnarray*}
 2^{k} &=&  1.2^{k-1}+...+1.2^{i_0+1}+0.2^{i_0}+(\alpha_{i_0-1}+\beta_{i_0-1}+\gamma_{i_0-1})2^{i_0-1}+...+2.2^{0}, \\
       &=& 1.2^{k-1}+...+0.2^{i_{0}}+3.2^{i_0-1}+...+2.2^0, \\
       &=& 1.2^{k-1}+...+1.2^{i_{0}}+1.2^{i_0-1}+...+2.2^0.
\end{eqnarray*}
The representation of $2^k$ would be valid until the term $i_0-1$ and we can repeat the argument with the next element whose coefficient is different in the representation of $2^k$ in (\ref{4}) and (\ref{5}).
\item if $\alpha_{i_0-1}+\beta_{i_0-1}+\gamma_{i_0-1}=2$, then
\begin{eqnarray*}
 2^{k} &=& 1.2^{k-1}+...+1.2^{i_0+1}+0.2^{i_0}+(\alpha_{i_{0}-1}+\beta_{i_{0}-1}+\gamma_{i_{0}-1})2^{i_{0}-1}+...+2.2^{0}, \\
       &=& 1.2^{k-1}+...+0.2^{i_{0}}+2.2^{i_0-1}+...+2.2^0, \\
       &=& 1.2^{k-1}+...+1.2^{i_{0}}+0.2^{i_0-1}+...+2.2^0,
       \end{eqnarray*}
This way we get that the representations of $2^k$ are equal until the term $i_0$. Next, to recover the term $2^{i_0-1}$ then $\alpha_{i_0-2}+\beta_{i_0-2}+\gamma_{i_0-2}=2$ or $3$ so that we are in cases (b) or (c)  and we repeat  until we get the decomposition (\ref{5}).
\end{enumerate}
\end{enumerate}
\end{description}

The analysis above justifies that the only possible cases in any representation   of $2^k$ in the form $(2^k-r)+\; (r-s) +\; s$ and each of the addends ($2^k-r$, $r-s$ and $s$) in the basis $B=\{2^l\}$, $l=0,\ldots,k-1$ are those that correspond to cases 1 or 2 in the statement of the lemma.

Let $m$ denote the number of inequalities in any of the systems  (\ref{sys1}) or (\ref{sys1.2}). First of all, we observe that  the last inequality has a common form in any of the systems, namely $x^2\le w_md$. Indeed, if $\alpha_{k-1}+\beta_{k-1}+\gamma_{k-1}=0$ then we shall consider the inequality
  \begin{equation} \label{ineq:l1}
  x^{2} \leq w_{m}w_{m-1}
 \end{equation}
  otherwise i.e if $\alpha_{k-1}+\beta_{k-1}+\gamma_{k-1}=1$ then we shall consider the inequality
  \begin{equation} \label{ineq:l2} x^{2} \leq w_{m}u^{\alpha_{k-1}}t^{\beta_{k-1}}x^{\gamma_{k-1}}.
  \end{equation}

Based on the above observation, we have that  the systems (\ref{sys1}) and (\ref{sys1.2})
always include (\ref{ineq:l1}) or (\ref{ineq:l2}) and other inequalities depending on the cases. Let us analyze the two cases.
\begin{description}
  \item[\textit{Case 1.}] Let $(x,t,u)$ be a solution of system (\ref{00}) and $\alpha_{i}+\beta_{i}+\gamma_{i}=1$ for all $0<i<k-1$. Set $w_1=\sqrt{ u^{\alpha_0} t^{\beta_0} x^{\gamma_0}}$ and
        $w_{i+1}=\sqrt{w_i u^{\alpha_i} t^{\beta_i} x^{\gamma_i}}$, $i=2,\ldots,k-2$. Clearly, $(x,t,u,w)$ is a solution of system (\ref{sys1}).

        Conversely, if $(x,t,u,w)$ is a solution of system (\ref{sys1}) then propagating backward from the  last inequality to the first one we prove that $(x,t,u)$ is also a feasible solution of (\ref{00}).

        Finally, it is clear that in this case, $m$, the number of inequalities necessary to represent (\ref{00}) as system (\ref{sys1}) is $m=k-1$.
  \item[\textit{Case 2.}]
   Let $(x,t,u)$ be a solution of system (\ref{00}) and $\alpha_{i}+\beta_{i}+\gamma_{i}$ for all $0<i<k-1$ satisfying the hypotheses of Item \textit{2.} in the thesis of the lemma. Set $w_1=\sqrt{ u^{\alpha_0} t^{\beta_0} x^{\gamma_0}}$ and
        $w_{i+1}$ for $i=2,\ldots,m$ being defined recursively according to the inequalities in (\ref{ineq:l2}) from the previous values of $w_j$, $j=1,\ldots,i$, and $u,t,x$. Clearly, $(x,t,u,w)$ is a solution of system (\ref{sys1}).

        Conversely, if $(x,t,u,w)$ is a solution of system (\ref{ineq:l2}) then propagating backward from the  last inequality to the first one we prove that $(x,t,u)$ is also a feasible solution of (\ref{00}).
        \medskip

        We conclude the proof observing that the number of inequalities $m$ in any of the two representations is fixed and it is equal to $m=1+2\#\{i:\alpha_{i}+\beta_{i}+\gamma_{i}\geq2, 1<i<k\}+\#\{i:\alpha_{i}+\beta_{i}+\gamma_{i}=1, 1<i<k\}\le 2k$.

\end{description}

\fin
\end{proof}
We illustrate the application of the above lemma with the following example. \\
\begin{example}
Let us consider $\tau=\frac{100000}{70001}$ which in turns means that $r=10^5$ and $s=70001$.
\begin{eqnarray*}
 x^{100000}&\leq & u^{70001}t^{29999}, \\
  x^{2^{17}}=x^{131072} &\leq & u^{70001}t^{29999}x^{31072}.
\end{eqnarray*}
The representations of the exponents of $u,t,x$ in the inequality above in power of 2 summands are:
\begin{eqnarray*}
  u:\quad70001 &=& 1.2^{16}+0.2^{15}+0.2^{14}+0.2^{13}+1.2^{12}+0.2^{11}+0.2^{10}+0.2^{9}+1.2^{8}+\\
  & & 0.2^{7}+1.2^{6}+1.2^{5}+1.2^{4}+0.2^{3}+0.2^{2}+0.2^{1}+1.2^{0} \\
  t:\quad29999 &=& 0.2^{16}+0.2^{15}+1.2^{14}+1.2^{13}+1.2^{12}+0.2^{11}+1.2^{10}+0.2^{9}+1.2^{8}+\\
  & & 0.2^{7}+0.2^{6}+1.2^{5}+0.2^{4}+1.2^{3}+1.2^{2}+1.2^{1}+1.2^{0} \\
  x:\quad31072 &=& 0.2^{16}+0.2^{15}+1.2^{14}+1.2^{13}+1.2^{12}+1.2^{11}+0.2^{10}+0.2^{9}+1.2^{8}+\\
   & & 0.2^{7}+1.2^{6}+1.2^{5}+0.2^{4}+0.2^{3}+0.2^{2}+0.2^{1}+0.2^{0}
\end{eqnarray*}
From the above decomposition, we realize that this example falls in case 2 and we obtain $c=3$. The table below shows the corresponding indexes of the $w$-inequalities of each bloc $i_j$, $j=1,2,3.$
\begin{center}
\begin{tabular}{|c|c|c|}
  \hline
  $i_1=5$ & $i_2=8$ & $i_3=12$ \\\hline
  $i_{l(1)}=7$ & $i_{l(2)}=9$ & $i_{l(3)}=15$ \\\hline
  $\theta(1)=6$ & $\theta(2)=11$ & $\theta(3)=16$ \\
  \hline
\end{tabular},
\end{center}
the total number of inequalities is $m=1+2*6+9=22.$\\

Then we get the decomposition
\medskip

\noindent
\begin{tabular}{ccccc}
\hline
  level\; 1 & level\; 2 & level\; 3 & level\; 4 & level\; 5\\\hline
  $w_{1}^{2} \leq ut$ & $w_{2}^{2} \leq w_{1}t$ & $w_{3}^{2} \leq w_{2}t$& $w_{4}^{2} \leq w_{3}t$ & $w_{5}^{2} \leq w_{4}t$ \\
   & &  &  & \\
\end{tabular}

\noindent
\begin{tabular}{ccc}
  & Bloc $i_1$ & \\\hline
  level\; 6 & level\; 7 & level\; 8\\\hline
  $ w_{6}^{2} \leq ut $& $w_{8}^{2} \leq w_{6}u$ & $w_{10}^{2} \leq w_{8}w_{9}$\\
  $w_{7}^{2} \leq w_{5}x$ & $w_{9}^{2} \leq w_{7}x $ &\\
  & & \\
\end{tabular}

\noindent
\begin{tabular}{ccc}
& Bloc $i_2$ & \\\hline
level\; 8 & level\; 9 & level\; 10 \\\hline
 $w_{10}^{2} \leq w_{8}w_{9}$ & $w_{11}^{2} \leq ut$ & $w_{13}^{2} \leq w_{11}w_{12}$  \\
 & $w_{12}^{2} \leq w_{10}x$ & \\
 & &\\
\end{tabular}

\noindent
\begin{tabular}{cc}
\hline
level\; 11 & level\; 12 \\\hline
 $w_{14}^{2} \leq w_{13}t$ & $w_{15}^{2} \leq w_{14}x$ \\
  & \\
\end{tabular}

\noindent
\begin{tabular}{cccc}
& & \hspace*{-1.9cm}Bloc $i_3$ & \\\hline
level\; 13 & level\; 14 & level\; 15 & level\; 16\\\hline
$w_{16}^{2} \leq ut$ & $w_{18}^{2} \leq w_{16}t$ & $w_{20}^{2} \leq w_{18}t$ & $w_{22}^{2} \leq w_{20}w_{21}$\\
$w_{17}^{2} \leq w_{15}x$ &$w_{19}^{2} \leq w_{17}x$ & $w_{21}^{2} \leq w_{19}x$ &  \\
& & & \\
\end{tabular}

\begin{tabular}{c}
\hline
level\; 17\\\hline
$ x^{2} \leq w_{22}u$\\
\\
\end{tabular}

\noindent From that set of inequalities one can easily obtain the original inequality by expanding backward, starting from the last level (level 17). Indeed,
\medskip

\noindent
\begin{tabular}{cccc}
\hline
  level\; 17 & level\; 16 & level\; 15 & level\; 14 \\\hline
  $x^{2} \leq w_{22}u$ & $x^{2^{2}} \leq u^2w_{20}w_{21}$ & $x^{2^{3}} \leq u^4txw_{18}w_{19}$& $x^{2^{4}} \leq u^8t^3x^3w_{16}w_{17}$ \\
   & & &\\
\end{tabular}

\noindent
\begin{tabular}{ccc}
\hline
  level\; 13 & level\; 12 & level\; 11 \\\hline
  $x^{2^{5}} \leq u^{17}t^7x^7w_{15}$ & $ x^{2^{6}} \leq u^{34}t^{14}x^{15}w_{14} $& $x^{2^{7}} \leq u^{68}t^{29}x^{30}w_{13}$  \\
  & & \\
\end{tabular}

\noindent
\begin{tabular}{ccc}
\hline
 level\; 10 & level\; 9 & level\; 8 \\\hline
  $x^{2^{8}} \leq u^{136}t^{58}x^{60}w_{11}w_{12}$ & $x^{2^{9}} \leq u^{273}t^{117}x^{121}w_{10}$ & $x^{2^{10}} \leq u^{546}t^{234}x^{242}w_{8}w_{9}$  \\
  & & \\
\end{tabular}

\noindent
\begin{tabular}{ccc}
\hline
 level\; 7 & level\; 6 & level\; 5 \\\hline
  $x^{2^{11}} \leq u^{1093}t^{468}x^{485}w_{6}w_{7}$ & $x^{2^{12}} \leq u^{2187}t^{937}x^{971}w_{5}$ & $x^{2^{13}} \leq u^{4375}t^{1874}x^{1942}w_{4}$ \\
  & & \\
\end{tabular}

\noindent
\begin{tabular}{cccc}
\hline
level\; 4 & level\; 3 & level\; 2 \\\hline
  $x^{2^{14}} \leq u^{8750}t^{3749}x^{3884}w_{3}$ & $x^{2^{15}} \leq u^{17500}t^{7499}x^{7768}w_{2}$ & $x^{2^{16}} \leq u^{35000}t^{14999}x^{15536}w_{1}$\\
  & & &\\
\end{tabular}
\noindent \\
\begin{tabular}{c}
\hline
level\; 1\\\hline
$ x^{2^{17}} \leq u^{70001}t^{29999}x^{31072}$
\end{tabular}

\end{example}

\begin{remark}
The particular case of the Euclidean norm ($\tau=2$) leads to a simpler representation based on a direct application of Schur complement.

%
%

\noindent Observe that the constraint $z_{i}^{2}\geq\omega_{i}^{2}\|x-a_{i}\|_{2}^{2}=\omega_{i}^{2}\displaystyle\sum_{j=1}^{d}(x_{j}-a_{ij})^{2},\;i=1,...,n$ can be written as  $L_i\succeq 0$, being
$$
L_{i}=\left(
  \begin{array}{cccc}
    z_{i}-(x_{1}-a_{i1})  & x_{2}-a_{i2} & \cdots & x_{d}-a_{id} \\
    x_{2}-a_{i2}  & z_{i}+(x_{1}-a_{i1}) &        & 0 \\
    \vdots &       & \ddots &  \\
    x_{d}-a_{id}  & 0     &        & z_{i}+(x_{1}-a_{i1}) \\
  \end{array}
\right).
$$
(Recall that for a symmetric matrix $A$, $A\succeq 0$ means $A$ to be positive semidefinite.)
%
%
\end{remark}

\bigskip

Next, based on Lemma \ref{le:n2} we can state the final representation result for the family of convex ordered continuous single facility location problems.

\begin{theorem} \label{t:teo2}
For any set of lambda weights satisfying $\lambda_{1}\geq . . . \geq\lambda_{n}$ and $\tau=\frac{r}{s}$  such that $r,s\in\N\setminus\{0\}$, $r>s$ and $\gcd(r,s)=1$, Problem \eqref{pb1} can be represented as a semidefinite programming problem with  $n^2+n(2d+1)$ linear constraints and at most $4nd \log r$ positive semidefinite constraints.
\end{theorem}
\begin{proof}
Using Theorem \ref{th:rep}  we have that Problem (\ref{pb1}) is equivalent to
\begin{eqnarray}\label{pb:th2}
        \min& \displaystyle \sum_{k=1}^{n}v_{k}+\sum_{i=1}^{n}w_{i} \\
        s.t &  v_{i}+w_{k}\geq\lambda_{k}z_{i} , & \forall i,k=1,...,n, \label{in:pbth2-1}  \\
& y_{ij}-x_j+a_{ij}\ge 0,& \forall i=1,...,n,\; j=1,...,d.\\
& y_{ij}+x_j-a_{ij} \ge 0,& \forall i=1,...,n,\; j=1,...,d.\nonumber\\
& y_{ij}^r\leq u_{ij}^{s}z_{i}^{r-s},& \forall i=1,...,n,\; j=1,...,d, \label{in:pbth2-4}  \\
& \omega_{i}^{\frac{r}{s}}\sum_{j=1}^d u_{ij}\le z_i,& \forall i=1,...,n,\;  \label{in-pbth2-5} \\
& u_{ij}\ge 0,& \forall\; i=1,...,n,\;  j=1,\ldots,d. \label{in-pbth2-6}
\end{eqnarray}
Then, we use Lemma \ref{le:n2}  to represent each one of the inequalities (\ref{in:pbth2-4}) for each $i,j$, as a system of at most $2\log r$ inequalities of the form (\ref{sys1}) or (\ref{sys1.2}). Next, we observe that all the inequalities that appear in (\ref{sys1}) or (\ref{sys1.2}) are of the form $a^2\le bc$, involving 3 variables, $a,b,c$ with $b,c$ non negative. Finally, it is well-known   by Schur complement that $$ a^2\le bc \quad \Leftrightarrow \left(\begin{array}{ccc} b+c  & 0 & 2a \\ 0 & b+c & b-c \\ 2a & b-c & b+c \end{array} \right) \succeq 0, \; b+c\ge 0.$$ Hence, Problem \eqref{pb1} is a SDP because it has a linear objective function, $n^2+n(2d+1)$ linear inequalities and at most $2nd\log(r)$ linear matrix inequalities.
\end{proof}
\fin

There is an interesting observation that follows from  the above result. It was already known that continuous convex ordered location problems with $\ell_1$ norm were reducible to linear programming (see e.g. \cite{NP05}). This paper proves that most continuous convex ordered location problems with $\ell_p$ norms  are reducible to SDP programming showing the similarities existing between all this class of problems and moreover that convex continuous single facility location problems are among the \textit{``easy''} optimization problems.
Moreover, Theorem \ref{t:teo2} allows us to apply the general theory of SDP to derive a general result of convergence for solving the  family of continuous convex ordered single facility location problems: Problem \eqref{pb1} is polynomially solvable in dimension $d\in \mathbb{N}$ and for any set of nonincreasing lambda weights. Moreover, we can be more precise and can state the following result:

\begin{theorem}
Let $\varepsilon>0$ be a prespecified accuracy and $(X^0,S^0)$ be a feasible primal-dual pair of initial solutions of Problem (\ref{pb:th2}-\ref{in-pbth2-6}). An optimal primal-dual pair $(X,S)$ satisfying $X\cdot S\le \varepsilon$ can be obtained in at most $O(\alpha \log \frac{X^0\cdot S^0}{\varepsilon})$ iterations and the complexity of each iteration is bounded above by $O(\alpha \beta^3,\alpha^2 \beta^2,\alpha^2)$ being $\alpha=3n+2nd(1+\log r)$ and $\beta=p$, the dimension  of the dual matrix variable $S_p$.
\end{theorem}

The reader may observe that this result is mainly of theoretical interest because the bound is based on general results on primal-dual algorithms, such as the modification of Kouleai and Terlaki \cite{KT07} to the  Mehrota  type algorithm \cite{Meh92} applied to SDP problems. Nevertheless, it is important to realize that it states an important difference with respect to any other known result in the area of continuous location where convergence results, when available, are only proven  limit of sequences and never in finite number of steps nor accuracy ensured. In this case, one can ensure a prespecified accuracy of the solution in a known number of iterations.

\section{Constrained Ordered median problems\label{s:constrained}}

This section extends the above results to constrained location problems. Therefore, we address now  the restricted case of Problem (\ref{pb1}). Let $\{g_{1}, \ldots, g_l\} \subset \mathbb{R}[x]$ be real polynomials and $\mathbf{K}:=\{x\in \mathbb{R}^{d}: g_{j}(x)\geq 0,\: j=1,\ldots ,l \}$ a basic
closed, compact semialgebraic set with nonempty interior satisfying the Archimedean property. Recall that the Archimedean property is equivalent to impose that for some $M>0$ the quadratic polynomial $u(x)=M-\sum_{i=1}^d x_i^2$ has a representation on $\mathbf{K}$ as $u\,=\,\sigma _{0}+\sum_{j=1}^{\ll}\sigma _{j}\,g_{j}$, for some  $\{\sigma_0, \ldots, \sigma_l\}\subset \mathbb{R}[x]$ being each $\sigma_j$ sum of squares \cite{putinar}. We remark that the assumption on the Archimedean property is not restrictive at all, since any semialgebraic  set $\mathbf{K}\subseteq \R^d$ for which is known that $\sum_{i=1}^d x_i^2 \leq M$ holds for some $M>0$ and for all $x\in \mathbf{K}$, admits a new representation $\mathbf{K'} = \mathbf{K} \cup \{x\in \mathbb{R}^{d}: g_{l+1}(x):=M-\sum_{i=1}^d x_i^2\geq 0\}$ that trivially verifies  the Archimedean property. In our framework the compactness assumption which is usually assumed in location analysis implies that this condition always holds.

In this framework we assume that the domain $\mathbf{K}$ is  compact and has nonempty interior. We observe that we can extend the  results in Section \ref{s:2} to a broader class of convex constrained problems.

In order to do that we need to introduce some notation. Let $\mathbf{\kappa}=({\kappa}_\alpha)$ be a real sequence
indexed in the monomial basis $(x^\beta z^\gamma v^\delta w^\zeta u^\eta y^\theta)$ of $\mathbb{R}[x,z,v,w,u,y]$ (with $%
\alpha=(\beta,\gamma,\delta,\zeta,\eta,\theta)\in\mathbb{N}^d\times\N^{n}\times\N^{n}\times\N^{n}
\times\N^{n\times d}\times\N^{n\times d}$). Let $D=3n+(2n+1)d$ denote the dimension of the space of variables. Define $\Upsilon= (x,z,v,w,u,y)$ to be the vector of indeterminates so that $\Upsilon^\alpha=x^\beta z^\gamma v^\delta w^\zeta u^\eta y^\theta$. For any integer $N$  consider the monomial vector
$$ [\Upsilon^N]=[(x,z,v,w,u,y)^N]=[1, x_1,\ldots x_d, z_1, \ldots z_n, \ldots, y_{nd}, x_1^2, x_1x_2 ,\ldots, y_{nd}^N]^t.
$$
Then, $[\Upsilon^N][\Upsilon^N]^t$ is a square matrix and we write
$$[\Upsilon^N][\Upsilon^N]^t=\sum_{0\le |\alpha| \le 2N} A_\alpha^0 \Upsilon^\alpha$$
for some symmetric $0/1$ matrices $A^0_\alpha$. Here, for a vector $\alpha$, $|\alpha|$ stands for the sum of its components.

For any sequence, $\mathbf{\kappa}=(\kappa_\alpha)_{\alpha \in \N^{D}}\subset\mathbb{R}$, indexed in the canonical monomial basis $\B$,  let $\L_\mathbf{\kappa}:\mathbb{R}[\Upsilon]\to\mathbb{R}$ be the linear functional defined, for any $f=\sum_{\alpha\in\mathbb{N}^d}f_\alpha\,\Upsilon^\alpha \in \R[\Upsilon]$, as $\L_\mathbf{\kappa}(f) :=
\sum_{\alpha}f_\alpha\,\kappa_\alpha$.

The \textit{moment} matrix $\Mo_N(\mathbf{\kappa})$ of order $N$ associated with $\mathbf{\kappa}$, has its rows and
columns indexed by $(\Upsilon^\alpha)$ and $\Mo_N(\mathbf{\kappa})(\alpha,\alpha')\,:=\,\L_{\mathbf{\kappa}}(\Upsilon^{\alpha+\alpha'})\,=\,\kappa_{%
\alpha+\alpha'}$, for $\vert\alpha\vert,\,\vert\alpha'\vert\,\leq N.$
Therefore,
$$ \Mo_N(\mathbf{\kappa})=\sum_{0\le |\alpha| \le 2N} A_\alpha^0 \kappa_\alpha$$

Note that the moment matrix of order $N$ has dimension ${{D+N} \choose {D} }\times {{D+N} \choose {D}}$ and that there are ${{D+2N} \choose {D}}$ $\mathbf{\kappa}_\alpha$ variables.

For $g\in \mathbb{R}[\Upsilon] \,(=\sum_{\nu \in \N^{M}}  g_{\nu} \Upsilon^\nu$), the \textit{localizing%
} matrix $\Mo_N(g \mathbf{\kappa})$ of order $N$ associated with $\mathbf{\kappa}$
and $g$, has its rows and columns indexed by $(\Upsilon^\alpha)$ and $\Mo_N(g\mathbf{\kappa})(\alpha,\alpha'):=\L_\mathbf{\kappa}(\Upsilon^{\alpha+\alpha'}%
g(\Upsilon))=
\sum_{\nu}g_\nu \kappa_{\nu+\alpha+\alpha'}$, for $\vert\alpha\vert,\vert\alpha'\vert\,\leq N$. Therefore,
$$ \Mo_N(g\mathbf{\kappa})=\sum_{0\le |\alpha| \le 2N} A_\alpha^g \kappa_\alpha,$$
for some symmetric $0/1$ matrices $A^g_\alpha$ that depend on the polynomial $g$. Also for convenience, we shall denote by $A_{e_i}^g$ the matrix associated with $\kappa_{e_i}$ the  moment variable linked to the monomial $x^{e_i}=x_i^1$. (The interested reader is referred to \cite{lasserre22} and \cite{lasserrebook} for further details on the moment approach applied to global optimization.)


\begin{theorem} \label{t:convex}
Consider the restricted problem:
\begin{equation}\label{pb2}
\min_{x\in\mathbf{K}\subset\R^{d}}\sum_{i=1}^{n}\lambda_{i}\omega_{\sigma(i)}\|x-a_{\sigma(i)}\|_{\tau}.
\end{equation}
Assume that the hypothesis of Theorem \ref{t:teo2} holds. In addition, any of the following conditions hold:
\begin{enumerate}
\item $g_i(x)$ are concave for $i=1,\ldots,\ell$ and $-\sum_{i=1}^{\ell} \mu_i \nabla^2 g_i(x) \succ 0$ for each dual pair $(x,\mu)$ of the problem of minimizing any linear functional $c^tx$ on $\mathbf{K}$ (\textit{Positive Definite Lagrange Hessian} (PDLH)).
\item $g_i(x)$ are sos-concave on $\mathbf{K}$ for $i=1,\ldots,\ell$ or $g_i(x)$ are concave on $\mathbf{K}$ and strictly concave on the boundary of $\mathbf{K}$ where they vanish, i.e. $\partial \mathbf{K}\cap \partial \{x\in \mathbb{R}^d: g_i(x)=0\}$, for all $i=1,\dots,\ell$.
\item $g_i(x)$ are strictly quasi-concave on $\mathbf{K}$ for $i=1,\ldots,\ell$.
\end{enumerate}
Then, there exists a constructive  finite dimension embedding, which only depends on  $\tau$ and $g_i$, $i=1,\ldots,\ell$, such that (\ref{pb2}) is a semidefinite problem.
\end{theorem}
\begin{proof}
The unconstrained version of Problem (\ref{pb2}) can be equivalently written as a SDP using the result in Theorem \ref{t:teo2}. Therefore, it remains to prove that under the conditions 1, 2 or 3 the constraint set $x\in \mathbf{K}$ is also exactly represented as a finite number of semidefinite constraints or equivalently that it is semidefinite representable (SDr).

Let us begin with condition 1. Consider the system of linear matrix inequalities:
\begin{equation} \label{sys:C1}
A_0^{(k)}+\sum_{i=1}^{\ell} A_{e_i}^{g_k} x_i +\sum_{1\le \alpha\le 2N} A_{\alpha}^{g_k} \kappa_{\alpha}\succeq 0, \qquad k=0,\ldots,\ell.
\end{equation}
Under the hypothesis of condition 1, the set $\mathbf{K}$ satisfies the Putinar-Prestel's Bounded Degree Nonnegative Representation property (PP-BDNR), see \cite[Theorem 6]{HelNie2012}. This condition ensures that there exists a finite $N$ such that the set
$$ \hat S_N=\{(x,\kappa): \mbox{ satisfying inequalites (\ref{sys:C1})}\}$$
projects via the $x$ coordinate onto the set $\mathbf{K}$. Hence, an exact lifted representation of Problem (\ref{pb2}) is the one provided by Theorem \ref{t:teo2} augmented with the additional linear matrix inequalities in (\ref{sys:C1}).

Let us assume now that condition 2 holds. Consider the set\\
$$ \hat S_N=\{(x,\kappa): \Mo_{N}(\mathbf{\kappa})\succeq 0, \; L_{\mathbf{\kappa}}(g_i)\ge 0, \; i=1\ldots,\ell, \; L_{\mathbf{\kappa}}(x_j)=x_j,\; j=1,\ldots,d,\; \kappa_0=1\}.$$

Under our hypothesis, Theorem 11.11 in \cite{lasserrebook} ensures that there exists a finite $N$ such that $\hat S_N$ projects via the $x$ variables onto the set $\mathbf{K}$. Hence, we obtain another lifted exact SDP formulation for Problem (\ref{pb2}) using the formulation induced by Theorem \ref{t:teo2} augmented with the  inequalities $\Mo_{N}(\mathbf{\kappa})\succeq 0, \; L_\mathbf{\kappa}(g_i)\ge 0, \; i=1\ldots,\ell, \; L_\mathbf{\kappa}(x_j)=x_j,\; j=1,\ldots,d,\; \kappa_0=1$.

Finally, let us consider the case in condition 3. If $g_i$ are strictly quasi-concave on $\mathbf{K}$, Proposition 10 in \cite{HelNie2012} implies that one can find some new polynomials $-p_i$ that have positive definite Hessian in $\mathbf{K}$. Let us denote $P:=\{x\in \mathbb{R}^d: p_i(x)\ge 0,\; i=1,\dots,\ell\}$. Thus, in some open set $U$ containing $\mathbf{K}$ it holds $ P\cap U=\mathbf{K}$.

Next, define the set
$$ \hat S_N=\{(x,\mathbf{\kappa}): \mbox{ satisfying inequalites (\ref{sys:C3-1})-(\ref{sys:C3-4})}\}$$
where the set of   linear matrix inequalities (\ref{sys:C3-1})-(\ref{sys:C3-4}) are given by:
\begin{eqnarray}
A_0^{(k)}+\sum_{i=1}^{\ell} A_{e_i}^{(k)} x_i +\sum_{1\le \alpha\le 2N} A_{\alpha}^{p_k} \kappa_{\alpha}\succeq 0, & k=0,\ldots,\ell \label{sys:C3-1}\\
L_{\mathbf{\kappa}}(p_k)\ge 0, & k=0,\ldots,\ell \label{sys:C3-2}\\
L_{\mathbf{\kappa}}(x_j)=x_j, & j=1,\ldots,d \label{sys:C3-3}\\
\kappa_0=1. \label{sys:C3-4}&
\end{eqnarray}

Under the hypothesis of condition 3, Theorem 24  in \cite{HelNie2012} ensures that there exists a finite $N$ such that $\hat S_N$ projects via the $x$ variables onto $\mathbf{K}$. Hence, we obtain the third lifted exact SDP formulation for Problem (\ref{pb2}) using the formulation induced by Theorem \ref{t:teo2} augmented with the  inequalities (\ref{sys:C3-1})-(\ref{sys:C3-4}).

We observe that according to Theorem 29 in \cite{HelNie2012}, since we assume the Archimedean property holds in all these cases, $N$ can be bounded above by some finite constant that only depends on the polynomials $g_i$, $i=1,\ldots,\ell$.
\fin
\end{proof}

We shall finish this section with another convergence result applicable to the case of non-convex constrained location problems. Again, let $\{g_{1}, \ldots, g_l\} \subset \mathbb{R}[x]$ and $\mathbf{K}:=\{x\in \mathbb{R}^{d}: g_{k}(x)\geq 0,\: k=1,\ldots ,\ll \}$ a basic, compact,
closed semialgebraic set satisfying the Archimedean property, with nonempty interior and such that $\K$ does not satisfies the hypothesis of Theorem \ref{t:convex}, in particular some of the $g_j$ may not be concave.

Now, we can prove a convergence result that allows us to approximate, up to any degree of accuracy, the solution of the  class of problems defined in (\ref{pb2}) when the hypothesis of Theorem \ref{t:convex} fails.  Let $\xi_k:=\lceil (\mathrm{deg}\, g_k)/2\rceil$  where $\{g_1,\ldots,g_{\ell}\}$ are  the polynomial constraints that define $\mathbf{K}$. For $N\geq N_{0}:=%
\displaystyle\max \{\max_{k=1,\ldots,\ell} \xi_k,1\}$, we introduce the  following hierarchy of semidefinite
programs:

\begin{eqnarray}\label{pb:th5}
  (\mathbf{Q}_N):      \min& \displaystyle \sum_{k=1}^{n}v_{k}+\sum_{i=1}^{n}w_{i} \\
        s.t. &  v_{i}+w_{k}\geq\lambda_{k}z_{i} , & \forall i,k=1,...,n, \label{in:pbth5-1}  \\
& y_{ij}-x_j+a_{ij}\ge 0,& \forall i=1,...,n,\; j=1,...,d.\\
& y_{ij}+x_j-a_{ij} \ge 0,& \forall i=1,...,n,\; j=1,...,d.\nonumber\\
& y_{ij}^r\leq u_{ij}^{s}z_{i}^{r-s},& \forall i=1,...,n,\; j=1,...,d, \label{in:pbth5-4}  \\
& \omega_{i}^{\frac{r}{s}}\sum_{j=1}^d u_{ij}\le z_i,& \forall i=1,...,n,\; \\
& \Mo_{N}(\mathbf{\kappa})  \succeq 0, & \\
& \Mo_{N-\xi_k}(g_k,{\bf \kappa})  \succeq 0,& k=1,\ldots ,\ell, \\
& \qquad L_{\kappa}(x_j)=x_j, & j=1,\ldots,d, \nonumber \\
& \qquad L_{\kappa}(z_i)=z_i, & i=1,\ldots,n, \nonumber \\
& \qquad L_{\kappa}(v_i)=v_i, & i=1,\ldots,n, \nonumber \\
& \qquad L_{\kappa}(w_i)=w_i, & i=1,\ldots,n,  \nonumber \\
& \qquad L_{\kappa}(u_{ij})=u_{ij}, & i=1,\ldots,n, \;  j=1,\ldots,d,  \nonumber \\
& \qquad L_{\kappa}(y_{ij})=y_{ij}, & i=1,\ldots,n, \;  j=1,\ldots,d,  \nonumber \\
& \kappa_0=1 & \nonumber \\
& u_{ij}\ge 0,& \forall\; i=1,...,n,\;  j=1,\ldots,d.
\end{eqnarray}

with optimal value denoted $\min \mathbf{Q}_{N}$.

\begin{theorem}
\label{t:th5} Consider $\rho_{\lambda}$ defined as the optimal value of the problem:
\begin{equation}\label{pb5}
\rho_{\lambda}=\min_{x\in\mathbf{K}\subset\R^{d}}\sum_{i=1}^{n}\lambda_{i}\omega_{\sigma(i)}\|x-a_{\sigma(i)}\|_{\tau}.
\end{equation}
Then, with the notation above:

\textrm{(a)} $\min\mathbf{{Q}}_N\uparrow \rho_{\lambda}$ as $N\to\infty$.

\textrm{(b)} Let $\mathbf{\kappa}^N$ be an optimal solution of Problem $(\mathbf{Q}_N)$. If
\begin{equation*}  \label{finiteconv}
\mathrm{rank}\,\Mo_N(\mathbf{\kappa}^N)\,=\,\mathrm{rank}\,\Mo_{N-N_0}(\mathbf{\kappa}%
^r)\,=\,\vartheta
\end{equation*}
then $\min \mathbf{{Q}}_{N}=\rho_{\lambda}$ and one may extract $\vartheta$ points $$(x^*(i),z^*(i),v^*(i),w^*(i),u^*(i),y^*(i))_{i=1}^{\vartheta}\subset\mathbf{K},$$ all global minimizers of Problem (\ref{pb5}).
\end{theorem}
\begin{proof}
First of all, we observe that an optimal solution of Problem (\ref{pb5}) does exist by the compactness assumption on $\mathbf{K}$. Moreover, the convergence of the semidefinite sequence of problems $(\mathbf{Q}_N)$ follows from a result by  Lasserre \cite[Theorem 5.6]{lasserrebook} that it is applied here to the SDP problem (\ref{pb:th2}-\ref{in-pbth2-6}) on the closed semialgebraic set $\mathbf{K}$. The second assertion on the rank condition, for extracting optimal solutions, follows from applying \cite[Theorem 5.7]{lasserrebook} to Problem $(\mathbf{Q}_N)$.
\par \fin
\end{proof}

\definecolor{LightCyan}{rgb}{0.94,0.94,0.94}

\begin{table}[h]
\begin{center}
\begin{tabular}{|>{\columncolor{LightCyan}}l|>{\columncolor{LightCyan}}l|>{\columncolor{LightCyan}}r>{\columncolor{LightCyan}}r|>{\columncolor{LightCyan}}r>{\columncolor{LightCyan}}r|>{\columncolor{LightCyan}}r>{\columncolor{LightCyan}}r|}\hline
\hline
 \multicolumn{2}{|>{\columncolor{LightCyan}}c}{Dim} &  \multicolumn{2}{|>{\columncolor{LightCyan}}c}{2} &  \multicolumn{2}{|>{\columncolor{LightCyan}}c}{3} &  \multicolumn{2}{|>{\columncolor{LightCyan}}c|}{10}\\ \hline
 \rowcolor{LightCyan}\texttt{$\tau$} & \texttt{n}    &  \texttt{Time(Ave)}&  \texttt{Gap(Ave)} &  \texttt{Time(Ave)}&  \texttt{Gap(Ave)} &  \texttt{Time(Ave)}&  \texttt{Gap(Ave)} \\ \hline
\textbf{$\tau$=1.5} & 10    & 0.20	 & 0.00000001 & 0.28    & 0.00000001 & 0.86   & 0.00000000 \\
                    & 100   & 1.71	 & 0.00000001 & 3.16    & 0.00000001 & 10.89  & 0.00000001 \\
                    & 500   & 10.78	 & 0.00000001 & 15.84   & 0.00000001 & 51.23  & 0.00000000 \\
                    & 1000  & 21.22	 & 0.00000001 & 30.67   & 0.00000001 & 103.17 & 0.00000001 \\
                    & 5000  & 103.50 & 0.00000001 & 178.50  & 0.00000001 & 566.64 & 0.00000001 \\
                    & 10000 & 210.22 & 0.00000001 & 455.05  & 0.00000001 & 1330.36& 0.00000001 \\\hline
\textbf{$\tau$=2}   & 10    & 0.06 	 & 0.00000000 & 0.12    & 0.00000000 & 0.46   & 0.00000000 \\
                    & 100   & 0.40 	 & 0.00000000 & 0.99    & 0.00000000 & 4.63   & 0.00000000 \\
                    & 500   & 1.50	 & 0.00000000 & 5.83    & 0.00000000 & 21.63  & 0.00000000 \\
                    & 1000  & 3.27	 & 0.00000000 & 11.08   & 0.00000000 & 44.04  & 0.00000000 \\
                    & 5000  & 17.06	 & 0.00000000 & 58.16   & 0.00000000 & 218.82 & 0.00000000 \\
                    & 10000 & 33.19	 & 0.00000000 & 118.91  & 0.00000000 & 455.30 & 0.00000000 \\\hline
\textbf{$\tau$=3}   & 10    & 0.19	 & 0.00000001 & 0.31    & 0.00000001 & 0.99   & 0.00000001 \\
                    & 100   & 1.88	 & 0.00000001 & 3.60    & 0.00000001 & 12.71  & 0.00000001 \\
                    & 500   & 10.82	 & 0.00000001 & 17.87   & 0.00000001 & 57.91  & 0.00000001 \\
                    & 1000  & 21.73  & 0.00000001 & 33.99   & 0.00000001 & 118.11 & 0.00000001 \\
                    & 5000  & 110.87 & 0.00000001 & 181.17  & 0.00000001 & 646.46 & 0.00000001 \\
                    & 10000 & 245.66 & 0.00000001 & 477.38  & 0.00000001 & 1616.26& 0.00000001 \\\hline
\textbf{$\tau$=3.5} & 10    & 0.33	 & 0.00000001 & 0.47    & 0.00000001 & 1.75   & 0.00000001 \\
                    & 100   & 3.87	 & 0.00000001 & 5.44    & 0.00000001 & 19.71  & 0.00000001 \\
                    & 500   & 18.62  & 0.00000001 & 26.99   & 0.00000001 & 92.64  & 0.00000001 \\
                    & 1000  & 37.06  & 0.00000001 & 51.50   & 0.00000001 & 192.77 & 0.00000001 \\
                    & 5000  & 280.27 & 0.00000001 & 304.94  & 0.00000001 & 1178.17& 0.00000001 \\
                    & 10000 & 964.18 & 0.00000001 & 872.29  & 0.00000001 & 2431.58& 0.00000001 \\\hline
\multicolumn{8}{c}{}
\end{tabular}
    \caption{Computational results for Weber problem with different norms and different dimensions.\label{t:weber}}
    \end{center}
    \end{table}

\section{Computational Experiments}
\label{s:4}
A series of computational experiments have been performed in order
to evaluate the behavior of the proposed methodology.
Programs have been coded in MATLAB R2010b and executed in a PC with an
Intel Core i7 processor at 2x 2.93 GHz and 16 GB of RAM. The semidefinite programs have been solved by calling SDPT3 4.0\cite{sdpt3}.

We run the algorithm for several well-known continuous single facility convex ordered  location problems: Weber, center, k-center and general ordered median problem with random non-increasing monotone lambda. For each of them, we obtain the CPU times for computing solutions as well as the accuracy given by the solver SDPT3 4.0. In addition, to illustrate the application of the result in Theorem \ref{t:th5}, we also report results on a problem which consists of minimizing the range of distances in $\mathbb{R}^3$ with  two additional non-convex constraints. In this case, we include running times and gap with respect to upper bounds obtained with the battery of functions in \verb"optimset" of MATLAB   which only provide approximations on the exact solutions (optimality cannot be certified).

In this last case, in order to compute the accuracy of an obtained solution, we use the following measure for the error (see \cite{waki}):
\begin{equation} \label{eq:error}
\epsilon_{\rm obj} = \dfrac{|\text{the optimal value of the SDP } - {\rm fopt}|}{\max\{1, {\rm fopt}\}},
\end{equation}
where ${\rm fopt}$ is the approximated optimal value obtained with the functions in  \verb"optimset". The reader may note that we solve relaxed problems that give lower bounds. Therefore, the gap of our lower bounds is computed with respect to upper bounds which implies that actually may be even better than the one reported. (See Table \ref{t:range}.)

We have organized our computational experiments in three different problems types. Our test problems are set of points  randomly generated  on the $[0,10000]$  hypercubes of the $d$-dimensional space, $d=2,3,10$.  For  Weber, center and $k$-centrum problems, we could solve instances with at least 10000 points and for different $\ell_{\tau}$-norms, $\tau=1.5,2,3,3.5$. The general case with random lambda weights is harder and we only solved in all cases instances up to 1000 points.

Our goal is to present the results organized per problem type, framework space ($\mathbb{R}^d$, $d=2,3,10$) and norm ($\ell_{\tau}$,  $\tau=1.5,2,3,3.5$). Tables \ref{t:weber}, \ref{t:center} and \ref{t:kcentrum}  report our results on  the problems of minimizing the weighted sum of distances (Weber), the maximum distance (center) and the sum of the  $n/2$ largest distances ($n/2$-centrum). In all cases, the accuracy and resolutions times needed for the solver are rather good, even for 10000 points in dimension 10 and rather complicated norms (e.g. $\ell_{3.5}$). The reader can see that the hardest type is the $k$-centrum. In this problem type CPU times increase one order of magnitude because the structure of the problem does not allow to reduce the size of the formulation.

\begin{table}[h]
\begin{center}
\begin{tabular}{|>{\columncolor{LightCyan}}l|>{\columncolor{LightCyan}}l|>{\columncolor{LightCyan}}r>{\columncolor{LightCyan}}r|>{\columncolor{LightCyan}}r>{\columncolor{LightCyan}}r|>{\columncolor{LightCyan}}r>{\columncolor{LightCyan}}r|}\hline
 \multicolumn{2}{|>{\columncolor{LightCyan}}c|}{Dim} &  \multicolumn{2}{>{\columncolor{LightCyan}}c|}{2} &  \multicolumn{2}{>{\columncolor{LightCyan}}c|}{3} &  \multicolumn{2}{>{\columncolor{LightCyan}}c|}{10}\\\hline
 \rowcolor{LightCyan}\texttt{$\tau$} & \texttt{n}    &  \texttt{Time(Ave)}&  \texttt{Gap(Ave)} &  \texttt{Time(Ave)}&  \texttt{Gap(Ave)} &  \texttt{Time(Ave)}&  \texttt{Gap(Ave)} \\\hline
\textbf{$\tau$=1.5} & 10    & 0.22	 & 0.00000103 & 0.39   & 0.00000000 & 1.02   & 0.00010830 \\
                    & 100   & 2.13	 & 0.00000000 & 8.75   & 0.00000000 & 43.28  & 0.00002092 \\
                    & 500   & 12.59	 & 0.00000000 & 63.16  & 0.00000000 & 237.99 & 0.00000835 \\
                    & 1000  & 27.11	 & 0.00000000 & 115.35 & 0.00000000 & 327.04 & 0.00000615 \\
                    & 5000  & 150.34 & 0.00000001 & 357.49 & 0.00000000 & 1231.78& 0.00000539 \\
                    & 10000 & 371.39 & 0.00000000 & 1297.23& 0.00000000 & 2762.51& 0.00000714 \\\hline
\textbf{$\tau$=2}   & 10    & 0.11   & 0.00000005 & 0.17   & 0.00000001 & 0.44   & 0.00002446 \\
                    & 100   & 1.34   & 0.00000012 & 2.09   & 0.00000001 & 7.68   & 0.00024196 \\
                    & 500   & 8.80   & 0.00000008 & 13.69  & 0.00000001 & 50.25  & 0.00001027 \\
                    & 1000  & 20.85  & 0.00000051 & 30.27  & 0.00000001 & 115.64 & 0.00001137 \\
                    & 5000  & 119.90 & 0.00000133 & 212.21 & 0.00000001 & 912.44 & 0.00384177 \\
                    & 10000 & 287.13 & 0.00000265 & 467.08 & 0.00000001 & 1510.41& 0.00823965 \\\hline
\textbf{$\tau$=3}   & 10    & 0.23   & 0.00000021 & 0.37   & 0.00000001 & 1.08   & 0.00010908 \\
                    & 100   & 2.32   & 0.00000000 & 7.48   & 0.00000000 & 37.66  & 0.00001214 \\
                    & 500   & 14.47  & 0.00000001 & 52.27  & 0.00000018 & 209.46 & 0.00000624 \\
                    & 1000  & 28.93  & 0.00000001 & 119.96 & 0.00000056 & 293.48 & 0.00002443 \\
                    & 5000  & 160.96 & 0.00000001 & 456.11 & 0.00010340 & 1223.41& 0.00024982 \\
                    & 10000 & 434.79 & 0.00000000 & 6829.70& 0.00002703 & 2663.15& 0.00026867 \\\hline
\textbf{$\tau$=3.5} & 10    & 0.33	 & 0.00000000 & 0.56   & 0.00000570 & 1.80   & 0.00004732 \\
                    & 100   & 4.47	 & 0.00000013 & 13.72  & 0.00000036 & 68.91  & 0.00002253 \\
                    & 500   & 21.93	 & 0.00000002 & 90.65  & 0.00000020 & 373.21 & 0.00000236 \\
                    & 1000  & 44.82	 & 0.00000002 & 179.48 & 0.00001955 & 603.41 & 0.00001127 \\
                    & 5000  & 244.04 & 0.00000002 & 551.16 & 0.00000219 & 2279.93& 0.00020991 \\
                    & 10000 & 510.25 & 0.00000002 & 2618.87& 0.00001699 & 4814.26& 0.00006055 \\\hline
\multicolumn{8}{c}{}
\end{tabular}
    \caption{Computational results for the Center problem with different norms and different dimensions.\label{t:center}}
\end{center}
    \end{table}

\begin{table}[h]
\begin{center}
\begin{tabular}{|>{\columncolor{LightCyan}}l|>{\columncolor{LightCyan}}l|>{\columncolor{LightCyan}}r>{\columncolor{LightCyan}}r|>{\columncolor{LightCyan}}r>{\columncolor{LightCyan}}r|>{\columncolor{LightCyan}}r>{\columncolor{LightCyan}}r|}\hline
 \multicolumn{2}{|>{\columncolor{LightCyan}}c|}{Dim} &  \multicolumn{2}{>{\columncolor{LightCyan}}c|}{2} &  \multicolumn{2}{>{\columncolor{LightCyan}}c|}{3} & \multicolumn{2}{>{\columncolor{LightCyan}}c|}{10}\\\hline
 \rowcolor{LightCyan}\texttt{$\tau$} & \texttt{n}    &  \texttt{Time(Ave)}&  \texttt{Gap(Ave)} &  \texttt{Time(Ave)}&  \texttt{Gap(Ave)} &  \texttt{Time(Ave)}&  \texttt{Gap(Ave)} \\\hline
\textbf{$\tau$=1.5} & 10    & 0.29    & 0.00000001 & 0.42    & 0.00000000 & 1.01     & 0.00019138 \\
                    & 100   & 4.67    & 0.00000000 & 8.01    & 0.00000000 & 30.45    & 0.00000003 \\
                    & 500   & 39.92   & 0.00000001 & 49.26   & 0.00000000 & 205.37   & 0.00000002 \\
                    & 1000  & 68.11   & 0.00000001 & 109.52  & 0.00000000 & 437.95   & 0.00000002 \\
                    & 5000  & 476.95  & 0.00000000 & 668.08  & 0.00000000 & 4738.70  & 0.00000002 \\
                    & 10000 & 1242.57 & 0.00000001 & 2016.01 & 0.00000000 & 15348.57 & 0.00000002 \\\hline
\textbf{$\tau$=2}   & 10    & 0.13    & 0.00000000 & 0.18    & 0.00000000 & 0.45     & 0.00008221 \\
                    & 100   & 1.36    & 0.00000000 & 2.03    & 0.00000000 & 7.56     & 0.00000001 \\
                    & 500   & 8.40    & 0.00000000 & 15.63   & 0.00000000 & 53.87    & 0.00000000 \\
                    & 1000  & 22.38   & 0.00000000 & 30.90   & 0.00000000 & 108.22   & 0.00000000 \\
                    & 5000  & 128.18  & 0.00000000 & 195.56  & 0.00000000 & 815.34   & 0.00000000 \\
                    & 10000 & 337.17  & 0.00000000 & 460.98  & 0.00000000 & 2373.23  & 0.00000000 \\\hline
\textbf{$\tau$=3}   & 10    & 0.30    & 0.00000001 & 0.42    & 0.00000000 & 1.09     & 0.00003397 \\
                    & 100   & 5.65    & 0.00000001 & 10.20   & 0.00000000 & 40.08    & 0.00000016 \\
                    & 500   & 50.36   & 0.00000001 & 72.35   & 0.00000000 & 225.13   & 0.00000004 \\
                    & 1000  & 100.17  & 0.00000001 & 145.24  & 0.00000000 & 463.74   & 0.00000004 \\
                    & 5000  & 582.84  & 0.00000002 & 894.95  & 0.00000000 & 4067.13  & 0.00000002 \\
                    & 10000 & 1715.00 & 0.00000001 & 2565.21 & 0.00000002 & 13649.88 & 0.00000005 \\\hline
\textbf{$\tau$=3.5} & 10    & 0.44    & 0.00000006 & 0.60    & 0.00000000 & 1.81     & 0.00018065 \\
                    & 100   & 10.90   & 0.00000001 & 16.97   & 0.00000000 & 60.45    & 0.00000004 \\
                    & 500   & 80.28   & 0.00000002 & 124.50  & 0.00000002 & 379.20   & 0.00000004 \\
                    & 1000  & 171.62  & 0.00000002 & 252.79  & 0.00000002 & 852.59   & 0.00000004 \\
                    & 5000  & 1033.28 & 0.00000002 & 1700.71 & 0.00000002 & 8510.86  & 0.00000006 \\
                    & 10000 & 2345.25 & 0.00000002 & 4682.55 & 0.00000002 & 27723.99 & 0.00000004 \\\hline
\multicolumn{8}{c}{}
\end{tabular}
    \caption{Computational results for the 0.5-centrum problem with different norms and different dimensions.\label{t:kcentrum}}
\end{center}
    \end{table}

Table \ref{t:gomp} reports our results on the general ordered median problem with non-increasing monotone lambda weights. For this family of problems we could solve with our general formulation, in all cases, problem  sizes of 1000 points. Accuracy is rather good and the bottleneck here is the size of the SDP object to be handle since the fact that all lambda are non-null makes it impossible to simplify the representation.

\begin{table}[h]
\begin{center}
\begin{tabular}{|>{\columncolor{LightCyan}}l|>{\columncolor{LightCyan}}l|>{\columncolor{LightCyan}}r>{\columncolor{LightCyan}}r|>{\columncolor{LightCyan}}r>{\columncolor{LightCyan}}r|>{\columncolor{LightCyan}}r>{\columncolor{LightCyan}}r|}\hline
 \multicolumn{2}{|>{\columncolor{LightCyan}}c|}{Dim} &  \multicolumn{2}{>{\columncolor{LightCyan}}c|}{2} &  \multicolumn{2}{>{\columncolor{LightCyan}}c|}{3} & \multicolumn{2}{>{\columncolor{LightCyan}}c|}{10}\\\hline
 \rowcolor{LightCyan}\texttt{$\tau$} & \texttt{n}    &  \texttt{Time(Ave)}&  \texttt{Gap(Ave)} &  \texttt{Time(Ave)}&  \texttt{Gap(Ave)} &  \texttt{Time(Ave)}&  \texttt{Gap(Ave)} \\\hline
\textbf{$\tau$=1.5} & 10    & 0.24    & 0.00000000 & 0.40    & 0.00000000 & 1.19    & 0.00000256 \\
                    & 100   & 4.03    & 0.00000001 & 6.73    & 0.00000001 & 22.22   & 0.00000000 \\
                    & 500   & 159.42  & 0.00000001 & 190.77  & 0.00000001 & 380.99  & 0.00000000 \\
                    & 1000  & 1270.76 & 0.00000013 & 1730.61 & 0.00000002 & 2379.03 & 0.00000000 \\
\textbf{$\tau$=2}   & 10    & 0.14    & 0.00000003 & 0.18    & 0.00000001 & 0.62    & 0.00002783 \\
                    & 100   & 5.11    & 0.00000255 & 6.94    & 0.00000002 & 17.07   & 0.00000003 \\
                    & 500   & 427.54  & 0.00000314 & 443.47  & 0.00000266 & 1092.70 & 0.00000054 \\
                    & 1000  & 2079.97 & 0.00000315 & 7702.03 & 0.00000241 & 9235.59 & 0.00000073 \\
\textbf{$\tau$=3}   & 10    & 0.51    & 0.00000005 & 0.72    & 0.00000010 & 1.89    & 0.00000405 \\
                    & 100   & 64.14   & 0.00000483 & 58.10   & 0.00000119 & 152.45  & 0.00000148 \\
                    & 500   & 1532.17 & 0.00018236 & 2269.39 & 0.00003466 & 7950.09 & 0.00000931 \\
                    & 1000  & 4546.73 & 0.00025375 & 5678.17 & 0.00008893 & 18011.79& 0.00003434 \\
\textbf{$\tau$=3.5} & 10    & 0.63    & 0.00000130 & 1.48    & 0.00001077 & 7.05    & 0.00000197 \\
                    & 100   & 33.32   & 0.00000097 & 302.76  & 0.00009247 & 596.20  & 0.00024303 \\
                    & 500   & 1555.08 & 0.00000524 & 2774.06 & 0.00035570 & 8705.00 & 0.00014431 \\
                    & 1000  & 7625.95 & 0.00001702 & 7681.10 & 0.00059695 & 18845.92& 0.00020324 \\ \hline
\multicolumn{8}{c}{}
\end{tabular}
    \caption{Results for convex ordered median problem with general $\lambda$, different norms and dimensions.\label{t:gomp}}
\end{center}
    \end{table}

Finally, we also report in Table \ref{t:range}, for the sake of illustration,  an example of application of the result in Theorem \ref{t:th5}. This problem consists of the minimization of the difference between the maximum and minimum distances of a number of demand points ($n$ ranging between 10 and 1000) with respect to a solution point that must belong to a non-convex feasible region defined by the following two non-convex constraints  $x_{1}^2-2x_{2}^2-2x_{3}^2\geq0$ and $-2x_{1}^2+5x_{2}^2+4x_{3}^2\geq0$ within the unit cube. Clearly, this case is more difficult to solve since this problem is non-convex and thus, we need to resort to the hierarchy of relaxations introduced in Theorem  \ref{t:th5}. Nevertheless, we have obtained  good results in this case  even with the first relaxation order.

\begin{table}[h]
\begin{center}
\begin{tabular}{|>{\columncolor{LightCyan}}l|>{\columncolor{LightCyan}}l|>{\columncolor{LightCyan}}r>{\columncolor{LightCyan}}r|}\hline
 \multicolumn{2}{|>{\columncolor{LightCyan}}c|}{Dim} &  \multicolumn{2}{>{\columncolor{LightCyan}}c|}{3} \\\hline
 \rowcolor{LightCyan}\texttt{$\tau$} & \texttt{n}    &  \texttt{Time(Ave)}&  \texttt{Gap(Ave)} \\\hline
\textbf{$\tau$=2} & 10    & 0.46   & 0.00001623 \\
                  & 100   & 9.45   & 0.00457982 \\
                  & 500   & 80.56  & 0.00030263 \\
                  & 1000  & 204.96 & 0.00094492 \\ \hline
\multicolumn{4}{c}{}
\end{tabular}
    \caption{Computational results for the Range problem with two nonconvex constraints.\label{t:range}}
\end{center}
    \end{table}

\section{Conclusions \label{s:5}}
We develop a unified tool for minimizing convex ordered median location problems in finite dimension and with general $\ell_{\tau}$-norms.  We report computational results that show the powerfulness of this methodology to solve medium size continuous location problems.

This new approach solves a broad class of convex and non-convex continuous location problems that, up to date, were only partially solved in the specialized literature. We have tested this methodology with some medium size standard ordered median location problems in different dimensions and with different norms.

It is important to emphasize that one of the contributions of our approach is that the \textit{same algorithm} is used to solve all this family of location problems. This is an interesting novelty as compared with previous approaches, of course at the price of loosing some speed in the computations compared with some tailored algorithms for specific problems. Obviously, our goal was not to compete with previous algorithms since most of them are either designed for specific problems  or only applicable for planar problems. However, in all cases we obtained reasonable CPU times and accurate results. Furthermore, in many cases our running times for many problems could not be even compared with others since nobody had solved them before.

\section*{Acknowledgements}
The authors were partially supported by the project FQM-5849 (Junta de Andaluc\'ia$\backslash$FEDER). The first and third author were partially supported by the project  MTM2010-19576-C02-01 (MICINN, Spain). The first author was also supported by research group FQM-343.

\end{document}